\documentclass[12pt,a4paper]{amsart} 
\usepackage{enumerate}

\ifx\pdftexversion\undefined
\usepackage[dvips]{graphicx,color}
\else
\usepackage[pdftex]{graphicx,color}
\fi

\newcommand{\andname}{and}
\newcommand{\volumename}{volume}
\newcommand{\Inname}{in}
\newcommand{\ofname}{of}
\newcommand{\pagesname}{pages}
\newcommand{\deuxpoints}{:}

\let\ifanglais\iftrue

\def\R{{\mathbb R}}
\def\N{{\mathbb N}}

\def\vol{{\rm Vol}}

\newcommand{\ed}[1]{\textrm{d} #1}

\newtheoremstyle{mesthm}
  {10pt plus 1pt minus 1pt}
  {9pt plus 1pt minus 6pt}
  {\slshape}
  {0.5cm}
  {\bfseries}
  {.}
  {1ex}
  {}
\newtheoremstyle{mesdefi}
  {6pt plus 1pt minus 1pt}
  {6pt plus 1pt minus 1pt}
  {}
  {0.5cm}
  {\bfseries}
  {.}
  {1ex}
  {}%

\theoremstyle{mesthm}
\newtheorem{lema}{\ifanglais{\large L}emma\else{\large L}emme\fi}
\newtheorem{theo}[lema]{\ifanglais{\large T}heorem\else {\large
    T}h\'eor\`eme\fi}

\newtheorem{prop}[lema]{{\large P}roposition}  
\newtheorem{cor}[lema]{{\large C}orollary}

\newtheorem{conj}{\ifanglais{\large C}onjecture\else{\large C}onjecture\fi}

\theoremstyle{mesdefi}

\newtheorem{rmq}[lema]{\ifanglais{\large R}emark\else{\large
    R}emarque\fi}

\newcounter{step}

\DeclareMathOperator{\ent}{Ent}

\title[Asymptotic volume in Hilbert Geometries]%
{Asymptotic volume in Hilbert Geometries}
\author[C. Vernicos]{Constantin Vernicos}
\address{%
  Institut de math\'ematique et de mod\'elisation de Montpellier\\ 
  Universit\'e Montpellier 2 \\
  Case Courrier 051\\
  Place Eug\`ene Bataillon \\
  F--34395 Montpellier Cedex\\ 
  France} 
\email{Constantin.Vernicos@math.univ-montp2.fr}

\subjclass[2000]{53C60 (primary), 53C24, 58B20, 53A20 (secondary).}

\begin{document}

\begin{abstract}
We prove that the metric balls of a Hilbert geometry admit a volume growth at least polynomial of degree
their dimension.
We also characterise the convex polytopes as those having exactly polynomial volume growth of degree their dimension.
\end{abstract}

\maketitle

\section*{Introduction and statement of results}

We recall that Hilbert geometries are metric space defined in the interior
of a convex set using cross-ratios and as such are a generalisation of the
hyperbolic geometry. 

Among all Hilbert geometries, two families have emerged and play an important role.
On the one hand the polytopal ones, which for a given dimension are all
bi-lipschitz to the Hilbert geometry of the simplex (\cite{bernig,ver08,cvv3}), 
and on the other hand those whose boundary is
$C^2$ with positive Gaussian curvature, which are all bi-lipschitz to the Hyperbolic space (\cite{cpv}).

The present paper focuses on the volume of balls and was motivated by the following result due
to Burago and Ivanov\cite{bi95}:
\begin{theo}
  Let $(T^n,g)$ be a Riemannian torus, let $\omega_n$ be the euclidean volume of the euclidean
unit ball, and let $x$ be a point on the universal covering of $T^n$. Let also 
$B_g(x,r)$ be the metric ball of radius $r$ of the lifted metric centred at $x$. Then
\begin{itemize}
\item $\displaystyle{\text{Asvol(g)}=\lim_{r\to +\infty}\dfrac{\vol(B_g(x,r)}{r^n}\geq \omega_n}$;
\item Equality characterises flat tori.
\end{itemize}
\end{theo}
The author had the belief that such a statement may exists characterising
the simplexes among all Hilbert geometries. With that goal in mind we obtained a partial answer to that
question and a new characterisation of Polytopes in term of  their volume growth
as follows:

\begin{theo}
  There exists a constant $c_n>0$ such that for all Hilbert geometries $(\mathcal{C},d_{\mathcal C})$,
with $\mathcal C\subset \R^n$, for any point $p\in \mathcal{C}$ and any real number $r>0$, 
if one denotes by $\vol_{\mathcal{C}}$ the Busemann volume of $\mathcal{C}$, then we have
\begin{itemize}
\item $\vol_{\mathcal{C}}\bigl(B_\mathcal{C}(x,r)\bigr)\geq c_n r^n$;
\item The asymptotic volume is finite if and only if $\mathcal C$ is a polytope;
\item If $\mathcal C$ is a polytope with $k$ verticies, then one has
$$
\underline{{\rm Asvol}}(\mathcal{C})=\liminf_{r\to +\infty}\dfrac{\vol(B_{\mathcal C}(x,r)}{r^n}\geq c_n k\text{.}
$$
\end{itemize}
Therefore, in each dimension, there is only a finite number of families of polytopal Hilbert
geometries which may have an asymptotic volumes less than the simplex's.  
\end{theo}

This theorem is therefore weaker in its asymptotic results from the one
we expect, but in the meantime it gives the existence of an optimal lower bound on the volume growth
of balls which was not known. Indeed the previous result of this kind was obtained
by the author with Colbois \cite{ccv}, but it only gave a lower bound which converged
to zero as the radius of the ball went to infinity.

As a corollary we obtain a new proof of the fact that a Hilbert geometry
bi-lipschitz to a normed vector space is actually a polytopal one.

\section{Notations}A \textsl{proper} open set in $\R^n$ is a set not containing a whole line.

A Hilbert geometry
$(\mathcal{C},d_\mathcal{C})$ is a non empty \textsl{proper} open convex set $\mathcal{C}$
in $\R^n$ (that we shall call \textsl{convex domain}) with
the Hilbert distance 
$d_\mathcal{C}$ defined as follows: for any distinct points $p$ and $q$ in $\mathcal{C}$,
the line passing through $p$ and $q$ meets the boundary $\partial \mathcal{C}$ of $\mathcal{C}$
at two points $a$ and $b$, such that one walking on the line goes consecutively by $a$, $p$, $q$
$b$. Then we define
$$
d_{\mathcal C}(p,q) = \frac{1}{2} \ln [a,p,q,b],
$$
where $[a,p,q,b]$ is the cross ratio of $(a,p,q,b)$, i.e., 
$$
[a,p,q,b] = \frac{\| q-a \|}{\| p-a \|} \times \frac{\| p-b \|}{\| q-b\|} > 1,
$$
with $\| \cdot \|$ the canonical euclidean norm in
$\mathbb R^n$.  If either $a$ or $b$ is at infinity the corresponding ratio will be taken equal to $1$.

Note that the invariance of the cross ratio by a projective map implies the invariance 
of $d_{\mathcal C}$ by such a map.

These geometries are naturally endowed with
a  $C^0$ Finsler metric $F_\mathcal{C}$ as follows: 
if $p \in \mathcal C$ and $v \in T_{p}\mathcal C =\R^n$
with $v \neq 0$, the straight line passing by $p$ and directed by 
$v$ meets $\partial \mathcal C$ at two points $p_{\mathcal C}^{+}$ and
$p_{\mathcal C}^{-}$~. Then let $t^+$ and $t^-$ be two positive numbers such
that $p+t^+v=p_{\mathcal C}^{+}$ and $p-t^-v=p_{\mathcal C}^{-}$, in other words
these numbers corresponds to the time necessary to reach the boundary starting
at $p$ with the speed $v$ and $-v$. Then we define
$$
F_{\mathcal C}(p,v) = \frac{1}{2} \biggl(\frac{1}{t^+} + \frac{1}{t^-}\biggr) \quad \textrm{and} \quad F_{\mathcal C}(p , 0) = 0.
$$ 
Should $p_{\mathcal C}^{+}$ or
$p_{\mathcal C}^{-}$ be at infinity, then corresponding ratio will be taken equal to $0$.


The Hilbert distance $d_\mathcal{C}$ is the length distance associated to 
$F_{\mathcal C}$. We shall denote by $B_\mathcal{C}(p,r)$ the metric ball of radius $r$
centred at the point $p\in \mathcal{C}$.

Thanks to that Finsler metric, we can built two important Borel measures
$\mathcal C$. 

The first one is called the \textsl{Busemann} volume, will be denoted by $\vol_{\mathcal C}$
(It is actually the Hausdorff measure associated to the metric space $(\mathcal C,
d_{\mathcal C})$, see \cite{bbi}, example~5.5.13),  and is
defined as follows.
To any $p \in \mathcal C$, let $\beta_{\mathcal C}(p) = \{v \in \R^n ~|~ F_{\mathcal{C}}(p,v) < 1 \}$
be the open unit ball in
$T_{p}\mathcal{C} = \R^n$ of the norm $F_{\mathcal{C}}(p,\cdot)$ and 
$\omega_{n}$ the euclidean volume of the open unit ball of the standard euclidean space
$\R^n$.
Consider the (density) function $h_{\mathcal C}\colon  \mathcal C \longrightarrow \R$ given by $h_{\mathcal C}(p)
= \omega_{n}/\text{Leb}\bigl(\beta_{\mathcal C}(p)\bigr),$ where $\text{Leb}$ is the canonical Lebesgue measure
of $\R^n$ equal to $1$ on the unit "hypercube". 

$$
\vol_{\mathcal C}(A) = \int_{A} h_{\mathcal C}(p) \ed{\text{Leb}(p)}
$$
for any Borel set $A$ of $\mathcal C$.

The second one, called the \textsl{Holmes-Thompson} volume will be denoted
by $\mu_{HT,\mathcal{C}}$, and is defined as follows. Let $\beta_{\mathcal C}^*(p)$ be the polar
dual of $\beta_{\mathcal C}(p)$ and $h_{HT,\mathcal C}\colon\mathcal C \longrightarrow \R$  the density defined
by $h_{HT, \mathcal C}(p)
=\text{Leb}\bigl(\beta_{\mathcal C}^*(p)\bigr)/\omega_n,$. Then $\mu_{HT,\mathcal{C}}$ is the measure associated
to that density.

We can actually consider a wider family of measure as follows
Let ${\mathcal E}_n$ be the set of pointed properly open convex sets in $\R^n$. These are the pairs $(\omega,x)$, such that
$\omega$ is  a properly open convex set and $x$ a point inside $\omega$. 
We shall say that a function $f\colon {\mathcal E}_n\to \R^+\setminus\{0\}$
is a \textsl{proper density} if it is 
\begin{description}
\item[Continuous] with respect to the Hausdorff pointed topology on ${\mathcal E}_n$;
\item[Monotone decreasing] with respect to inclusion of the convex sets, i.e., if $x\in \omega \subset \Omega$
then $f(\Omega,x)\leq f(\omega,x)$.
\item[Chain rule compatible] if for any projective transformation $T$ one has
$$
f\bigl(T(\omega),T(x)\bigr) \text{Jac}(T)= f(\omega,x)\text{.} 
$$
\end{description}
We will say that $f$ is a \textsl{normalised proper density} if in addition
$f$ coincides with the standard Riemannian volume on the Hyperbolic geometry of ellipsoids.
Let us denote by $PD_n$ the set of proper densities over ${\mathcal E}_n$. 

Let us now recall a result of Benzecri~\cite{benzecri} which states that the action of the group of
projective transformations on ${\mathcal E}_n$ is co-compact. Then, as remarked by L.~Marquis,
for any pair $f,g$ of proper densities, there exists a constant $C>0$ ($C\geq 1$ for the normalised ones)
such that
that for any $(\omega,x)\in {\mathcal E}$ one has
\begin{equation}
  \label{eqmeasures}
  \frac1C \leq \frac{f(\omega,x)}{g(\omega,x)}\leq C\text{.}
\end{equation}

In the same way we defined the Busemann and the Holmes-Thompson volumes, 
to any proper density $f$ one can associate a Borel measure on $\mathcal C$
$\mu_{f,\mathcal C}$. Integrating the equivalence (\ref{eqmeasures}) we obtain that  for any pair $f,g$
of densities, there exists a constant $C>0$ such that for any Borel set $U\subset  \mathcal C$ we will have
\begin{equation}
  \label{eqmeasures2}
  \frac1C \mu_{g,\mathcal C }(U)\leq \mu_{f,\mathcal C}(U)\leq C \mu_{g,\mathcal C }(U)\text{.} 
\end{equation}
We shall call \textsl{proper measures with density} the family of
measures obtain this way.

To a proper density $f\in PD_{n-1}$ we can also associate a $n-1$-dimensional measure,
denoted by $\text{Area}_{f,\mathcal C}$,
on hypersurfaces in $\mathcal C$ as follows. Let $S_{n-1}$ be smooth a hypersurface,
and consider for a point $p$ in the hypersurface $S_{n-1}$ its tangent hyperplane $H(p)$,
then the measure will be given by 
\begin{equation}
  \label{eqarea}
  d\text{Area}_{f,{\mathcal C}}(p) = d\mu_{f,{\mathcal C }\cap H(p)}(p)\text{.}
\end{equation}

Let now $\mu_{f,{\mathcal C}}$ be  a proper measure with density over $\mathcal C$,
then the volume entropy of $\mathcal C$ is defined by
\begin{equation}
  \label{eqvolentropie}
  \ent(\mathcal C)=\liminf_{r\to +\infty} \dfrac{\ln \mu_{f,{\mathcal C}}\bigl(B_{\mathcal C}(p,r)\bigr)}{r} \text{.}
\end{equation}
This number does not depend on either $f$ or $p$.

\section{Lower bound}

\begin{theo}\label{lowerbound}
The volume growth of  balls and spheres in a Hilbert Geometry is at least polynomial.
More precisely  for any integer $n\in \N^*$ and any proper density $f\in PD_n$ (resp. $g\in PD_{n-1}$)
there exists  a constant $c_B(n,f)$ (resp. $c_S(n,g)$) 
such that given a $n$-dimensional Hilbert Geometry $(\mathcal{C},d_\mathcal{C})$ and a
point $x\in \mathcal{C}$, for any $r\in \R^+$ one has the following inequalities
\begin{eqnarray*}
c_B(n,f)r^n&\leq& \mu_{f,_\mathcal{C}}\bigl(B_\mathcal{C}(x,r)\bigr)\\
c_S(n,g)r^{n-1}&\leq & Area_{g,\mathcal{C}}\bigl(S_\mathcal{C}(x,r)\bigr)\text{.}
\end{eqnarray*}
\end{theo}

When considering the Busemann volume we will drop the $f$ or $g$ in the constants, i.e.,
we shall just denote the constants appearing Theorem \ref{lowerbound} by $c_B(n)$ and $c_S(n)$.
The constant associated to the Holmes-Thompson metric will be denoted by $c_B^*(n)$ and $c_S^*(n)$.

\begin{proof}
According to the inequality (\ref{eqmeasures2}) there exists a constant 
$C_1(n)$ such that one has the following comparison between
the Holmes\--Thompson and the Busemann measures: for any Borel set $U$ in $\mathcal{C}$,
\begin{equation}\label{equivmeasures}
  C_1^{-1}(n)\mu_{HT,\mathcal{C}}(U)\leq \vol_\mathcal{C}\bigl(U)\leq C_1(n)\mu_{HT,\mathcal{C}}(U)\text.
\end{equation}
hence our results will be true for either of these measures (and actually for
any proper density). 

Remark  that without loss of generality we can restrict to strictly convex
and $C^2$ convex sets, as the results passes to the limit with respect to the Hausdorff pointed topology.

Now let us do the proof by induction on the dimension for both measures at the same time.
First notice that the $1$-dimensional Hilbert Geometry is isometric to $\R$ thus, 
$c_B(1)=2=c_B^*(1)$ and we have actually an equality
for both measures.

Now suppose the result is true in dimension $n$ and let us prove that it holds in dimension $n+1$.
We need to consider a point $x\in \mathcal{C}\subset\R^{n+1}$ and the ball of radius $r$ centred at $x$.
Take a hyperplane $H$ intersecting the convex set $\mathcal{C}$ and containing $x$, 
by induction we thus have for any $s\in \R^+$
$$
c_B^*(n)s^n\leq\mu_{HT,\mathcal{C}\cap H}(B_{\mathcal{C}\cap H}(x,s))
$$
and we remark that $\mathcal{C}\cap H$ is totally geodesic, thus thanks to a Crofton formula valid in this setting (see \cite{alvarez_fernandez} theorem 1.1 and remark 2) or by minimality of totally geodesic submanifolds with respect 
to the  Holmes-Thompson measure (see \cite{alvarez_berck,berck}) we obtain that
the Holmes-Thompson area of the half spheres of radius $s$ centred at $x$ defined by $H$ have an area bigger or equal
to $B_{\mathcal{C}\cap H}(x,s)$, hence
$$
2c_B^*(n)s^n\leq\text{Area}_{HT,\mathcal{C}}\bigl(S_\mathcal{C}(x,s)\bigr)
$$
Which implies the result for the spheres.
Now thanks to the co-area inequality obtained in \cite{berck_Bernig_Vernicos} (lemmata 2.12 and 2.13)
we have the existence
of a constant $C_2(n)$ such that
\begin{equation}
  \label{eqcoarea}
  2c_B^*(n)s^n\leq\text{Area}_{HT}\bigl(S_\mathcal{C}(x,s)\bigr)\leq C_2(n)\dfrac{\partial}{\partial s} \vol\bigl(B_\mathcal{C}(x,s)\bigr)
\end{equation}
Hence it suffices to integrate the inequalities (\ref{eqcoarea}) between $0$ and $r$ to obtain the desired result for the Busemann measure and thanks to the comparison (\ref{equivmeasures}) for the Holmes-Thompson measure.
\end{proof}

The previous proof also implies the following proposition
related to the volume entropy.

\begin{prop}
  The volume entropy of a Hilbert geometry is bigger or equal to any of
its lower dimensional sections, \textsl{i.e.}, let $(\mathcal{C},d_{\mathcal{C}})$
be a $n$-dimensional Hilbert geometry and let $A_k$ be an affine $k$-dimensional
subspace of $\R^n$, then we have
$$
\ent(A_k\cap\mathcal{C})\leq\ent(\mathcal{C})\text{.}
$$
\end{prop}
\begin{proof}
 We just do the proof for $k=n-1$, the general result easily follows.
Let  $H$ be a hyperplane such that $H\cap \mathcal{C}$ is an open $(n-1)$-dimensional
convex set and  $p$ be a point inside $H\cap \mathcal{C}$.
Then by minimality, as in the previous proof one has
$$
2\mu_{HT, H\cap \mathcal{C}}\bigl(B_{H\cap \mathcal{C}}(p,R)\bigr)\leq \text{Area}_{HT,\mathcal{C}}\bigl(S_\mathcal{C}(p,R)\bigr)\text{.}
$$ 
Now taking the logarithm of both sides, dividing by $R$ and taking the limit as
$R$ goes to infinity proves that $\ent(H\cap \mathcal{C})$ is lower
than the spherical volume entropy of $\mathcal{C}$, which is equal to the
volume entropy of $\mathcal{C}$ following \cite{berck_Bernig_Vernicos}.
\end{proof}

Let us now focus on the Busemann volume and define the $n$-asymptotic volume by
\begin{equation}
  \underline{{\rm Asvol}_n}(\mathcal{C},x)=\liminf_{r\to +\infty} \dfrac{\vol_\mathcal{C}\bigl(B_\mathcal{C}(x,r)\bigr)}{r^n}
\end{equation}

\begin{conj}
Let $b_n$ be the asymptotic volume of the simplex (which equals the euclidean volume of the unit ball
if the volume is the Busemann volume) then we have

\begin{enumerate}
\item $ \underline{{\rm Asvol}_n}(\mathcal{C})\geq b_n$;
\item with equality if and only if $(\mathcal{C},d_\mathcal{C})$ is a simplex.
\end{enumerate}
\end{conj}

In a previous paper we studied the volume entropy of Hilbert geometries \cite{berck_Bernig_Vernicos}.
In the present paper we are focusing on Hilbert geometries for which the entropy is equal to zero.

In that case one can focus on the polytopal entropy defined by
 \begin{equation}
   \label{eqpolyentropy}
   {\rm PolEnt}(\mathcal{C})=\liminf_{r\to +\infty}\dfrac{\ln\Bigl(\vol_\mathcal{C}\bigl(B_\mathcal{C}(x,r)\bigr)\Bigr)}{\ln r}\text{.}
 \end{equation}
 This number can be defined for any proper measure with density, and does not depend
 on the proper density nor on the centre $x$.

\section{Upper Bound}

In this section we will consider the Busemann volume and denote once more by $c_B(n)$ the constant
given by Theorem \ref{lowerbound}.

Let us define the upper $n$-asymptotic volume by
\begin{equation}
  \overline{{\rm Asvol}_n}(\mathcal{C},x)=\limsup_{r\to +\infty} \dfrac{\vol_\mathcal{C}\bigl(B_\mathcal{C}(x,r)\bigr)}{r^n}
\end{equation}

\begin{prop}\label{finite_upperasvol}
  Let $(\mathcal{C},d_\mathcal{C})$ be an $n$-dimensional Hilbert Geometry,
the upper $n$-asymptotic volume is finite, if and only if $\mathcal{C}$ is a polytope, \textsl{i.e.},
  \begin{equation*}
\overline{{\rm Asvol}_n}(\mathcal{C})=\limsup_{r\to +\infty} \dfrac{\vol_\mathcal{C}\bigl(B_\mathcal{C}(x,r)\bigr)}{r^n} < +\infty \iff \mathcal{C} \text{ is a polytope}
  \end{equation*}
\end{prop}

\begin{rmq}
  The results obtained in Vernicos~\cite{ver12} give for any $n\in \N$ the existence
  of convex sets in $\R^2$ with polytopal volume growth, such that $n+2\leq{\rm PolEnt}(\mathcal{C})\leq n+3$
which are therefore not polytopes. 
\end{rmq}

\begin{proof}[Proof of proposition \ref{finite_upperasvol}]
  In \cite{ver09} we proved that in a polytope of $\R^n$, the volume of balls of radius $r$ was less than a constant times $r^n$.
Once again the co-area inequality obtained in \cite{berck_Bernig_Vernicos} implies that the volume of spheres of radius $r$
is also bounded by a constant times $r^{n-1}$, otherwise we would get a contradiction.

Reciprocally, let us suppose that $\mathcal{C}$ is not a polytope. Hence by Krein-Millman's theorem
for any $k\in\N^*$ there exists a subset $X_k$ with $k$ extremal points of $\partial\mathcal{C}$. 

Then let us fix some $k$ and a corresponding subset $X_k$ of the boundary. Then for any given pair of points $(p,q)$ in $X_k$,
there exists an $R_{p,q}>0$ such that for any $R>R_{p,q}$, if $x_p(R)$ and $x_q(R)$ are the intersections
of the lines $(xp)$ and $(xq)$ with the sphere of radius $R$  centred at $x$, \textit{i.e.}, $S_\mathcal{C}(x,R)$ ,
then
$$
d_\mathcal{C}\bigl(x_p(R),x_q(R)\bigr)>R.
$$
(An easy computation, using the tangents at $p$ and $q$ to $\mathcal{C}$ in the plane $(opq)$
 shows that it is bounded from below by a quantity which is equivalent to $2R$ as $R$ goes to infinity).

Let us consider $R_k>\max\{ R_{p,q}\mid p,q \in X_k\}$ 

For any extremal point $p\in X_K$ let us consider the ball $B_{p,R}$ of radius $R/4$ centred at $x_p(3R/4)$.
Then for any radius $R>R_k$, and any pair of points $p,q\in X_k$, the corresponding balls $B_{p,R}$ and $B_{q,R}$
are disjoints. Thus
\begin{equation}
  \label{equnionball}
  \coprod_{p\in X_k}B_{p,R} \subset B_\mathcal{C}(x,R)\text{.}
\end{equation}

Hence for any $R>R_k$ the sum of the volume of the balls $B_{p,R}$, for $p\in X_k$,  
is smaller than the volume of the ball of radius $R$ centred at $x$, i.e., 
\begin{equation} 
\label{inequnionball} 
  \sum_{p\in X_k} \vol_\mathcal{C}\bigl(B_{p,R}\bigr)\leq \vol_\mathcal{C}\bigl(B_\mathcal{C}(x,R)\bigr) \text{.}
\end{equation}

We now apply the lower bound on the volume of the balls of radius $R/4$ obtained in theorem \ref{lowerbound}
  to the inequality (\ref{inequnionball}) to obtain a lower bound in terms of $k$ and $R^n$:
 \begin{equation}
   \label{ineqlowerwithk}
   k c_B(n) (R/4)^n \leq \vol_\mathcal{C}\bigl(B_\mathcal{C}(x,R)\bigr)\text{,}
 \end{equation}
and taking the limit as $R$ goes to infinity we finally get

\begin{equation}
  \label{ineqlowerasvol}
  k\dfrac{c_B(n)}{4^n}\leq \underline{\text{Asvol}_n}(\mathcal{C})\text{.}
\end{equation}

This being true for any integer $k\in \N$ we conclude that
 $\underline{\text{Asvol}_n}(\mathcal{C})$ is infinite.
\end{proof}

During the previous proof with ended up with the equation (\ref{ineqlowerasvol})
which can be summed up in the following way.

\begin{prop}
For any integer $n\in \N$, there exists a constant $a(n)$ such that for any
 polytope $\mathcal{P}_k$ with $k$ vertices and non-empty interior in $\R^n$ one has
$$
a(n)k \leq \underline{{\rm Asvol}_n}(\mathcal{C})\text{.}
$$ 
\end{prop}

We also get the following corollary, which is also a consequence of Colbois-Verovic \cite{cpv2}

\begin{cor}
  Let $(\mathcal{C},d_\mathcal{C})$ be a Hilbert geometry in $\R^n$ bi-lipschitz equivalent
to a $n$-dimensional vector space, then $\mathcal{C}$ is a polytope.
\end{cor}
\begin{proof}
  This implies that there is a constant $c$ such that the volume of a ball of 
radius $r$ is less $c\cdot r^n$, henceforth the asymptotic volume is finite.
\end{proof}



\end{document}